\documentclass[12pt]{amsart}
\usepackage[dvipdfm]{graphicx} 
\usepackage[initials]{amsrefs}
	
\usepackage{amsthm, amsmath, amssymb,graphicx}

\numberwithin{equation}{section}

\theoremstyle{plain}	     
\newtheorem{thm}{Theorem}[section] 

\newtheorem{lem}[thm]{Lemma}
\newtheorem{prop}[thm]{Proposition}
\theoremstyle{definition}

\theoremstyle{remark} 
\newtheorem{rem}[thm]{Remark}

\newcommand{\vp}{\varphi}

\begin{document}

\title{Remarks on Redheffer's inequality}

\author{Nagi Suzuki}
\address[Nagi Suzuki]{Department of Mathematical Sciences, Shibaura Institute of Technology, 307 Fukasaku, Minuma-ku, Saitama-shi, Saitama 337-8570, Japan}

\author{Shingo Takeuchi}
\address[Shingo Takeuchi]{Department of Mathematical Sciences, Shibaura Institute of Technology, 307 Fukasaku, Minuma-ku, Saitama-shi, Saitama 337-8570, Japan}
\email[Corresponding author]{shingo@shibaura-it.ac.jp}
\thanks{The work of S. Takeuchi was supported by JSPS KAKENHI Grant Number 22K03392.}


\begin{abstract}
Redheffer's inequality and its generalization are applied to the study 
of behavior and estimates of the first eigenvalue of $p$-Laplacian with respect to $p$. Furthermore, a Redheffer-type inequality for 
the generalized trigonometric function is extended to a broader class.
\end{abstract}


\maketitle


\section{Introduction}

Redheffer's inequality
\begin{equation}
\label{eq:redheffer}
\frac{\pi^2-x^2}{\pi^2+x^2} \leq \frac{\sin{x}}{x}
\end{equation}
was proposed by R. Redheffer in 1968 as an advanced problem
in Amer. Math. Monthly \cite{JCWSRRZ1968}*{No. 5642}.
In another issue of the journal \cite{RW1969}*{No. 5642}, 
the inequality was titled \textit{A delightful 
inequality} and J.P. Williams gave a proof relying on the infinite product 
representation of the sine function.
Since then, Redheffer's inequality has been widely studied in the area of inequality
(see \cites{W1972,Z2009,ZS2008,LZ2011,SB2015,BS2016,OT2021} and the references given there).

Research on Redheffer's inequality primarily consists of work aimed at improving its precision as an inequality and discovering similar inequalities involving trigonometric functions, hyperbolic functions, or special functions. See for instance \cites{B2007,ZS2008,Z2009,OT2021}. To the authors' knowledge, there have been few applications of this delightful inequality. 

The purpose of this paper is to apply Redheffer's inequality 
and its generalization by Zhu and Sun \cite{ZS2008} 
to the study of eigenvalue problems for differential equations, specifically to  
the analysis of behavior and estimates of the first eigenvalue of $p$-Laplacian with respect to $p$. 
As a result, we can provide an alternative proof for the monotonicity of the first eigenvalue shown by Kajikiya, Tanaka and Tanaka \cite{KTT2017}, 
while also refining the estimate for the first eigenvalue obtained by Kajikiya and Takeuchi \cite{KT2025}.
Furthermore, we extend the Redheffer-type inequality for generalized cosine function, previously discovered by Ozawa and Takeuchi \cite{OT2021}, to a broader class.


\section{Preliminaries}


In this section, we define generalized trigonometric functions and 
generalized $\pi$ necessary for stating the results in the next section.

Let $p,q \in (1,\infty)$ and 
$$F_{p,q}(x):=\int_0^x \frac{dt}{(1-t^q)^{1/p}}, \quad x \in [0,1].$$
We will denote by $\sin_{p,q}$ the inverse function of $F_{p,q}$, i.e.,
$$\sin_{p,q}{x}:=F_{p,q}^{-1}(x).$$
Clearly, $\sin_{p,q}{x}$ is an increasing function in $[0,\pi_{p,q}/2]$ to $[0,1]$,
where
$$\pi_{p,q}:=2F_{p,q}(1)=2\int_0^1 \frac{dt}{(1-t^q)^{1/p}}
=\frac{2}{q}B\left(1-\frac{1}{p},\frac{1}{q}\right),$$
where $B$ is the beta function.
We extend $\sin_{p,q}{x}$ to $(\pi_{p,q}/2,\pi_{p,q}]$ by $\sin_{p,q}{(\pi_{p,q}-x)}$
and to the whole real line $\mathbb{R}$ as the odd $\pi_{p,q}$-antiperiodic 
continuation of the function. 
Since $\sin_{p,q}{x} \in C^1(\mathbb{R})$,
we also define $\cos_{p,q}{x}$ by $\cos_{p,q}{x}:=(\sin_{p,q}{x})'$,
where ${}':=d/dx$.
Then, it follows that 
$$|\cos_{p,q}{x}|^p+|\sin_{p,q}{x}|^q=1.$$
In case $(p,q)=(2,2)$, it is obvious that $\sin_{p,q}{x},\ \cos_{p,q}{x}$ 
and $\pi_{p,q}$ are reduced to the ordinary $\sin{x},\ \cos{x}$ and $\pi$,
respectively. 
This is the reason why these functions and the constant are called
\textit{generalized trigonometric functions} 
and \textit{generalized $\pi$} (by parameter $(p,q)$), respectively. 

We denote $\sin_{p,p}{x},\ \cos_{p,q}{x}$ and $\pi_{p,p}$ by
$\sin_p{x},\ \cos_p{x}$ and $\pi_p$, respectively.
In particular, it is noted that
\[\pi_p
=\frac{2}{p}B\left(1-\frac{1}{p},\frac{1}{p}\right)
=\frac{2\pi}{p\sin{(\pi/p)}}.\]

We conclude this section by explaining the utility of generalized 
trigonometric functions and generalized $\pi$.
Let $\lambda(p)$ denote the first eigenvalue of $p$-Laplacian
with zero Dirichlet boundary conditions on the interval $[-1,1]$, 
and let $\vp$ denote the corresponding 
eigenfunction with $\max_{-1 \leq x \leq 1}\vp(x)=1$.
That is, $\lambda(p)$ and $\vp$ satisfy
\[-(|\vp'|^{p-2}\vp')'=\lambda(p) |\vp|^{p-2}\vp, \quad \vp(-1)=\vp(1)=0.\]
Then, it is known that $\lambda(p)$ and $\vp$ can be expressed 
as follows (see, for example, \cites{E1981,DM1999}):
\begin{gather}
\lambda(p)=(p-1)\left(\frac{\pi_p}{2}\right)^p
=(p-1)\left(\frac{\pi}{p\sin{(\pi/p)}}\right)^p,
\label{eq:ev} \\
\vp(x)=\sin_p{\left(\frac{\pi_p}{2}(x+1)\right)}.
\notag
\end{gather}
Using $\lambda(p)$ and $\vp$, 
all solutions $(\lambda,u)$ to the eigenvalue problem
\[-(|u'|^{p-2}u')'=\lambda |u|^{p-2}u, \quad u(a)=u(b)=0\]
can be fully written down as follows:
\[\lambda_k=\left(\frac{2k}{b-a}\right)^p \lambda(p),\quad 
u_k(x)=C\vp\left(\frac{2k(x-a)}{b-a}-1\right),\]
where $k=1,2,3,\ldots$ and $C$ is an arbitrary constant.
Thus, $\lambda(p)$ and $\vp$ provide us with a representation 
of the solutions to the eigenvalue problem for $p$-Laplacian.

%
%
%
%

\section{Main results}

This section presents the three main results of this paper.
The first two are applications of Redheffer's inequality to 
the analysis of the first eigenvalue $\lambda (p)$ of $p$-Laplacian, 
while the remaining one is an extension of a known Redheffer-type inequality.

\subsection{Alternative proof of the monotonicity of $\lambda(p)$}

Kajikiya, Tanaka, and Tanaka \cite{KTT2017} show that the 
behavior of the first eigenvalue of the $p$-Laplacian on the 
interval $[-L,L]$ with respect to $p$ changes significantly 
at the threshold $L=1$. Here, we present their result for $L=1$.
The first eigenvalue of the $p$-Laplacian in this case is 
the $\lambda(p)$ given by 
\eqref{eq:ev}. 

\begin{prop}[\cite{KTT2017}]
\label{thm:KTT}
$\lambda(p)$ is strictly increasing with respect to $p$
in $(1,\infty)$.
\end{prop}

Below, we provide an alternative proof of Proposition 
\ref{thm:KTT} using Redheffer's inequality \eqref{eq:redheffer}.

First, let us explain their proof idea in \cite{KTT2017}. 
They prove Proposition \ref{thm:KTT} using standard differential calculus. 
Ultimately, the proof reduces to the problem of determining the sign 
of a certain function. Although they do not explicitly state it as a 
lemma, we will present it here in the form of a lemma for convenience.

\begin{lem}
\label{lem:KTT}
Let 
\[k(x):=x^2(\pi-x)^2-(\pi^2-\pi x+x^2)\sin^2{x}.\]
Then, $\lambda'(p)>0$ for all $p \in (1,\infty)$ 
if $k(x)<0$ for all $x \in (0,\pi)$.
\end{lem}

By this lemma, to prove Proposition \ref{thm:KTT}, 
it suffices to show that $k(x) < 0$.
We briefly review their (skillful) proof in \cite{KTT2017} that $k(x) < 0$.
By the symmetry of $k(x)$, we may assume $0 < x < \pi/2$. Then, since $\sin x > x - x^3/6$,
\[k(x)<-\frac{x^3}{36}(x^5-\pi x^4-(12-\pi^2)x^3+12\pi x^2-12\pi^2 x+36\pi).\]
Let the quintic expression on the right-hand side be $K(x)$. 
Furthermore, since
\begin{align*}
K'(x)
&=5x^4-4\pi x^3-3(12-\pi^2)x^2+24\pi x-12\pi^2\\
&=-(4\pi-5x)x^3-3(12-\pi^2)x^2-12\pi(\pi-2x)<0
\end{align*}
holds for $x \in (0,\pi/2)$, we obtain
\[K(x)>K\left(\frac{\pi}{2}\right)=\frac{3}{32}\pi^5-\frac{9}{2}\pi^3+36\pi=2.25844\cdots>0.\]
Therefore, $k(x)<-x^3K(x)/36<0$.
Thus, Proposition \ref{thm:KTT} follows from Lemma \ref{lem:KTT}.

However, in fact, using Redheffer's inequality \eqref{eq:redheffer}, 
we can immediately see that $k(x) < 0$ as follows.
%
Dividing both sides of $k(x)$ by $x^2$, we have
\[\frac{k(x)}{x^2}
=(\pi-x)^2-(\pi^2-\pi x +x^2)\left(\frac{\sin{x}}{x}\right)^2.\]
Using \eqref{eq:redheffer} to the right-hand side yields
\begin{align*}
\frac{k(x)}{x^2} 
&\leq (\pi-x)^2-(\pi^2-\pi x +x^2)\left(\frac{\pi^2-x^2}{\pi^2+x^2}\right)^2\\
&=-\frac{\pi x (\pi-x)^4}{(\pi^2+x^2)^2}<0;
\end{align*}
hence, $k(x)<0$ for $x \in (0,\pi)$.
%

\subsection{Improvement to the estimate of $\lambda(p)$}

Regarding the first eigenvalue of the $p$-Laplacian on an $N$-dimensional ball, 
Huang \cite{H1997} and Benedikt and Dr\'{a}bek \cite{BD2012} obtained upper and lower bounds expressed 
as polynomials in $p$. Applying their bounds for $N=1$ yields the following result:

\begin{prop}[\cites{H1997,BD2012}]
\label{thm:BD2012}
For any $p \in (1,\infty)$, it holds that
\[p \leq \lambda(p) \leq p+1.\]
\end{prop}

When $p$ is close to $1$, the lower bound can be improved (see \cites{BD2013,BEZ2009}). However, here we prefer the simplest estimate above.

In fact, the upper bound of the inequality in Proposition \ref{thm:BD2012}
is not optimal even when restricted to linear functions.
Indeed, it has recently been improved 
by Kajikiya and Takeuchi \cite{KT2025} as follows:

\begin{prop}[\cite{KT2025}]
\label{thm:KT2025}
For any $p \in (1,\infty)$, it holds that
\begin{equation*}
p <\lambda(p) <p+\frac{\pi^2}{6}-1.
\end{equation*}
Moreover, the upper bound is optimal in the following sense:
\begin{equation}
\label{eq:evopt}
\lim_{p \to \infty}(\lambda(p)-p)=\frac{\pi^2}{6}-1.
\end{equation}
\end{prop}

Now, we can further adjust the lower bound as follows to accommodate this upper bound.

\begin{thm}
\label{thm:eigenvalue}
For any $p \in (1,\infty)$, it holds that
\begin{equation}
\label{eq:egineq}
p+\frac{\pi^2}{6}-1-\frac{\pi^2}{6(p+1)}
<\lambda(p)
<p+\frac{\pi^2}{6}-1.
\end{equation}
Therefore, \eqref{eq:evopt} immediately follows.
\end{thm}

Below, we will explain the proof of Theorem \ref{thm:eigenvalue}.
The proof relies on the following 
Redheffer-type inequality by Zhu and Sun \cite{ZS2008}:

\begin{lem}[\cite{ZS2008}]
\label{lem:ZS2008}
Let $0<x<\pi$. Then,
\[\left(\frac{\pi^2-x^2}{\pi^2+x^2}\right)^\beta
\leq \frac{\sin{x}}{x} \leq 
\left(\frac{\pi^2-x^2}{\pi^2+x^2}\right)^\alpha\]
holds if and only if $\alpha \leq \pi^2/12$ and $\beta \geq 1$.
\end{lem}

\begin{proof}[Proof of Theorem \ref{thm:eigenvalue}]
Since the right-hand side of \eqref{eq:egineq} 
has been obtained by Proposition \ref{thm:KT2025}, 
only the left-hand side of \eqref{eq:egineq} is proved here.

We use the Redheffer-type inequality from Lemma \ref{lem:ZS2008} with $\alpha=\pi^2/12$:
for $x \in (0,\pi)$,
\[\frac{\sin{x}}{x} \leq \left(\frac{\pi^2-x^2}{\pi^2+x^2}\right)^{\frac{\pi^2}{12}}.\]
Putting $x=\pi/p\ (p>1)$, we obtain
\[\frac{\pi}{p\sin{(\pi/p)}}
\geq \left(\frac{p^2+1}{p^2-1}\right)^{\frac{\pi^2}{12}}.\]
Thus,
\[\lambda(p) \geq (p-1)
\left(1+\frac{2}{p^2-1}\right)^{\frac{\pi^2p}{12}}.\]
Using the inequality $(1+x)^\theta>1+\theta x$ for $x>0$ if $\theta>1$, 
we have
\[\lambda(p)> p-1+\frac{\pi^2p}{6(p+1)}
=p+\frac{\pi^2}{6}-1-\frac{\pi^2}{6(p+1)}\]
if $p>12/\pi^2$.
On the other hand, if $1< p \leq 12/\pi^2$, it is easily seen that $p > p + \pi^2/6 - 1-\pi^2/(6(p+1))$. Therefore, together with the fact that $\lambda(p) > p$ for any $p \in (1,\infty)$, \eqref{eq:egineq} is proved.
\end{proof}

\begin{rem}
Using the binomial expansion of $(1+x)^\theta$ 
up to higher-order terms yields increasingly accurate approximations, but here we present the simplest form using only the first-order term.
\end{rem}

\subsection{Extension of Redheffer's inequality to generalized trigonometric functions}

Ozawa and Takeuchi \cite{OT2021} 
generalized Redheffer's inequality so that it can be applied 
to a class of antiperiodic functions including the sine function.
To be precise, they establish a Redheffer-type inequality for a 
function $S$ in $(-\infty,\infty)$ that satisfies the following conditions:
there exist an $a \in (0,\infty)$ and a finite subset 
$P \subset (0,a)$, which may be the empty set, such that
\begin{enumerate}
\item[(S1)] $S(-x)=-S(x)$ and $S(a+x)=-S(x)$ for $x \in [0,\infty)$;
\item[(S2)] $0<S(x)<x$ for $x \in (0,a)$;
\item[(S3)] $S \in C([0,a]) \cap C^1([0,a)) \cap C^2([0,a) \setminus P)$;
\item[(S4)] $S'(x)^2-S''(x)S(x) \geq 1$ for $x \in [0,a) \setminus P$.
\end{enumerate}

It is clear that $S(x)$ is odd, continuous, piecewise smooth 
in $\mathbb{R}$ and 
antiperiodic with period $a$ in $[0,\infty)$; $S(na)=0$, 
$(-1)^nS(x)>0$ for $x \in (na,(n+1)a)$ and $n \in \mathbb{Z}$;
and $S(x)<x$ for $x \in (0,\infty)$ 
and $S \in C^1((-a,a))$ with $S'(0)=1$.
A typical example is $a=\pi,\ P=\emptyset$ and $S(x)=\sin{x}$.

\begin{prop}[\cite{OT2021}]
\label{thm:GRI}
Let $S$ be a function satisfying the conditions {\rm (S1)--(S4)}. Then,
\begin{equation}
\label{eq:gri}
\frac{a^2-x^2}{a^2+x^2} \leq \frac{S(x)}{x}.
\end{equation}
\end{prop}

It is worth pointing out that Redheffer's inequality \eqref{eq:redheffer} 
follows immediately from 
Proposition \ref{thm:GRI} with $a=\pi,\ P=\emptyset$ and $S(x)=\sin{x}$.

If we restrict the domain of function $S$ to $[0,a]$ and assume
\begin{itemize}
\item[(S1')] $S(0)=S(a)=0$
\end{itemize}
instead of (S1),
then \eqref{eq:gri} holds in this interval by a strict inequality
as follows.
The proof can be understood by reading the proof of Proposition \ref{thm:GRI}, i.e., \cite{OT2021}*{Theorem 1.1}, restricted to the interval $[0,a]$.

\begin{prop}
\label{cor:GRI}
Let $S$ be a function satisfying the conditions {\rm(S1')} and {\rm (S2)--(S4)}. Then,
\begin{equation*}
\frac{a^2-x^2}{a^2+x^2} < \frac{S(x)}{x}, \quad x \in (0,a).
\end{equation*}
\end{prop}

Applying Proposition \ref{thm:GRI} to $\sin_{p,q}{x}$, we can prove 
the following inequalities.

\begin{prop}[\cite{OT2021}]
\label{thm:GRI2}
Let $p,q \in [2,\infty)$. Then,
\begin{equation}
\label{eq:GRI}
\frac{\pi_{p,q}^2-x^2}{\pi_{p,q}^2+x^2} \leq \frac{\sin_{p,q}{x}}{x}.
\end{equation}
In particular, for $p \in [2,\infty)$,
\begin{equation*}
\frac{\pi_{p}^2-x^2}{\pi_{p}^2+x^2} \leq \frac{\sin_{p}{x}}{x}.
\end{equation*}
\end{prop}

From now on, we will consider the Redheffer-type inequality concerning the cosine function.
Since 
\[\frac{\sin{x}}{x}=\frac{2\sin{(x/2)}\cos{(x/2)}}{x}<\cos{\frac{x}{2}}\]
for $x \in (0,\pi)$,
it follows from \eqref{eq:redheffer} that 
\begin{equation}
\label{eq:cos}
\frac{\pi^2-x^2}{\pi^2+x^2} <\cos{\frac{x}{2}}, \quad x \in (0,\pi).
\end{equation}

Chen, Zhao and Qi \cite{CZQ2004} also directly prove \eqref{eq:cos} 
using the infinite product representation of the cosine function.

Similarly, a certain type of multiple-angle formula is also known for 
generalized sine functions
 (see Takeuchi \cite{T2016}). 
The formula gives
\[\frac{\sin_{2,q}{(2^{2/q-1}x)}}{(2^{2/q-1}x)}
=\frac{2\sin_{q^*,q}{(x/2)}\cos_{q^*,q}^{q^*-1}{(x/2)}}{x}
<\cos_{q^*,q}^{q^*-1}{\frac{x}{2}}\]
for $x \in (0,\pi_{q^*,q})=(0,\pi_{2,q}/(2^{2/q-1}))$,
where $q^*:=q/(q-1)$ for $q>1$.
Combining this inequality
and \eqref{eq:GRI}
yields the following generalization of \eqref{eq:cos}.

\begin{prop}[\cite{OT2021}]
\label{thm:cos}
Let $q \in [2,\infty)$. Then,
\begin{equation}
\label{eq:cosineq}
\frac{\pi_{q^*,q}^2-x^2}{\pi_{q^*,q}^2+x^2}<\cos_{q^*,q}^{q^*-1}{\frac{x}{2}}, 
\quad x \in (0,\pi_{q^*,q}).
\end{equation}
\end{prop}


Now, with all that in mind, we can finally present our finding.
Using Proposition \ref{cor:GRI}, we shall extend
the range of parameters in \eqref{eq:cosineq} as follows.

\begin{thm}
\label{thm:cosredheffer}
Let $p \in [q^*,\infty)$ and $q \in [2,\infty)$. Then,
\[\frac{\pi_{p,q}^2-x^2}{\pi_{p,q}^2+x^2}<\cos_{p,q}^{p-1}{\frac{x}{2}}, \quad x \in (0,\pi_{p,q}).\]
In particular, for $p \in [2,\infty)$,
\[\frac{\pi_p^2-x^2}{\pi_p^2+x^2}<\cos_p^{p-1}{\frac{x}{2}}, \quad x \in (0,\pi_p).\]
\end{thm}

Proposition \ref{thm:cos} is a corollary
of Theorem \ref{thm:cosredheffer} with $p=q^*$.


\begin{proof}[Proof of Theorem \ref{thm:cosredheffer}]
Let $a:=\pi_{p,q},\ P:=\emptyset$ and $S(x):=x\cos_{p,q}^{p-1}{(x/2)}$.
Then, (S1') is trivial. 
Since $0<\cos_{p,q}{(x/2)}<1$ in $(0,\pi_{p,q})$, 
$S$ satisfies (S2). Next, since $p>1$ and $q \geq 2$,
\begin{gather*}
S'(x)
=\cos_{p,q}^{p-1}\frac{x}{2}-\frac{(p-1)q}{2p}
x\sin_{p,q}^{q-1}\frac{x}{2},\\
S''(x)
=-\frac{(p-1)q}{p}\sin_{p,q}^{q-1}\frac{x}{2}
-\frac{(p-1)(q-1)q}{4p}x\sin_{p,q}^{q-2}\frac{x}{2}\cos_{p,q}\frac{x}{2}
\end{gather*}
are continuous in $[0,\pi_{p,q})$; hence, (S3) holds.

Below, we shall prove (S4), i.e., $S'(x)^2-S''(x)S(x) \geq 1$ for $x \in [0,\pi_{p,q})$. Set $f(x):=S'(x)^2-S''(x)S(x)-1$.
Since $f(0)=0$, it suffices to show $f(x) \geq 0$ for $x \in (0,\pi_{p,q})$. 
From the above expressions of $S'$ and $S''$,
\begin{multline*}
f(x)
=\cos_{p,q}^{2p-2}\frac{x}{2}\\
+\frac{(p-1)q}{4p}x^2\sin_{p,q}^{q-2}\frac{x}{2}
\left(\frac{(p-1)q}{p}\sin_{p,q}^q\frac{x}{2}
+(q-1)\cos_{p,q}^p\frac{x}{2}\right)-1.
\end{multline*}
Using $x/2 \geq \sin_{p,q}{(x/2)}$
and $\sin_{p,q}^q{(x/2)}=1-\cos_{p,q}^p{(x/2)}$
in $(0,\pi_{p,q})$,
we have
\begin{multline*}
f(x)
\geq \frac{(p-1)q}{p^2}
\left(\frac{p^2}{(p-1)q}\cos_{p,q}^{2p-2}\frac{x}{2}
+(p-1)q \right.\\
\left. +(2q-p-pq)\cos_{p,q}^p\frac{x}{2}
+(p-q)\cos_{p,q}^{2p}\frac{x}{2}
-\frac{p^2}{(p-1)q}\right).
\end{multline*}
Setting for $t \in [0,1]$,
\[g(t):=(p-q)t^{2p}+\frac{p^2}{(p-1)q}t^{2p-2}
+(2q-p-pq)t^p+(p-1)q-\frac{p^2}{(p-1)q},\]
we obtain
\[f(x) \geq \frac{(p-1)q}{p^2}g\left(\cos_{p,q}\frac{x}{2}\right).\]

To prove $f(x) \geq 0$ for $x \in (0,\pi_{p,q})$, it is sufficient to show $g(t) \geq 0$ for $t \in (0,1)$.
Differentiating $g(t)$, we have
\[g'(t)
=2p(p-q)t^{2p-1}
+\frac{2p^2}{q}t^{2p-3}
+p(2q-p-pq)t^{p-1}.\]
Putting $h(t):=t^{1-p}g'(t)$, i.e., for $t \in [0,1]$,
\[h(t)
:=2p(p-q)t^p+\frac{2p^2}{q}t^{p-2}+p(2q-p-pq),\]
we obtain
\begin{align*}
h'(t)
&=2p^2(p-q)t^{p-1}+\frac{2p^2(p-2)}{q}t^{p-3}
=2p^2t^{p-3}\left((p-q)t^2+\frac{p-2}{q}\right).
\end{align*}

We shall consider each of the following three cases:
(i) $p>q \geq 2$, (ii) $q \geq p \geq 2$
and (iii) $2>p \geq q^*$.

\textit{Case} (i).
In this case, $h'(t)>0$ for $t \in (0,1)$.
Then, 
\[h(t)<h(1)=-\frac{p^2}{q}(q+1)(q-2) \leq 0.\]
Hence, $g'(t)<0$ for $t \in (0,1)$.
Therefore, $g(t)>g(1)=0$ for $t \in (0,1)$.

\textit{Case} (ii).
In this case, for $t \in (0,1)$,
\begin{align*}
h(t)
<\frac{2p^2}{q}+p(2q-p-pq)
=-\frac{p}{q}(p(q-2)+(p-2)q^2)<0.
\end{align*}
Hence, $g'(t)<0$ for $t \in (0,1)$.
Therefore, $g(t)>g(1)=0$ for $t \in (0,1)$.

\textit{Case} (iii).
In this case, since $q>2>p$, $h'(t)<0$ for 
$t \in (0,1)$. 
If $q \leq p/(2-p)$, then 
\[h(t)<h(0)=-p(p-(2-p)q) \leq 0.\]
Hence, $g'(t)<0$ for $t \in (0,1)$.
Therefore, $g(t)>g(1)=0$ for $t \in (0,1)$.
We assume $q >p/(2-p)$. Then,
$h(0) >0$ and $h(1)<0$.
Since $h$ is strictly decreasing in $(0,1)$,
there exists $t_0 \in (0,1)$ such that 
$h(t)>0$ in $(0,t_0)$ and $h(t)<0$ in $(t_0,1)$;
hence, so $g'$ does. Then, $g$ is strictly increasing
in $(0,t_0)$ and strictly decreasing in $(t_0,1)$.
Thus, $g(t)>\min\{g(0),g(1)\}$.
It should be noted that $g(1)=0$ and 
\[g(0)=(p-1)q-\frac{p^2}{(p-1)q}
=(p^*+q)\left(\frac{p}{q^*}-1\right) \geq 0.\] 
Therefore, $g(t)>0$ for $t \in (0,1)$.
\end{proof}





\bibliographystyle{amsplain}
\bibliography{myreferences}

\end{document}